\documentclass{tran-l}

\usepackage{amsmath,amssymb,amsthm,url}

\usepackage{tikz-cd}
\usetikzlibrary{arrows}

\title[Topological rigidity with discrete singular set]{Topological rigidity and actions\\ on contractible manifolds\\ with discrete singular set}

\author[F Connolly]{Frank Connolly}
\address{Department of Mathematics, University of Notre Dame, Notre Dame IN 46556 USA}
\email{connolly.1@nd.edu}
\thanks{F Connolly was partly supported by Grant DMS-0601234 of the National Science Foundation.}

\author[J\,F Davis]{James F Davis}
\address{Department of Mathematics, Indiana University, Bloomington IN 47405 USA}
\email{jfdavis@indiana.edu}
\thanks{J\,F Davis was partly supported by Grant DMS-1210991 of the National Science Foundation.}

\author[Q Khan]{Qayum Khan}
\address{Department of Mathematics, Saint Louis University, St Louis MO 63103 USA}
\email{khanq@slu.edu}
\thanks{Q Khan was partly supported by Grant DMS-0904276 of the National Science Foundation.}


\copyrightinfo{2015}{American Mathematical Society}

\newtheorem{thm}{Theorem}[section]

\newtheorem{cor}[thm]{Corollary}
\newtheorem{lem}[thm]{Lemma}
\newtheorem{prop}[thm]{Proposition}
\newtheorem*{hyp}{Main Hypotheses}

\theoremstyle{definition}

\newtheorem{rem}[thm]{Remark}
\newtheorem*{undefn}{Definition}

\DeclareMathAlphabet{\matheurm}{U}{eur}{m}{n}

\newcommand{\all}{\matheurm{all}}

\newcommand{\fbc}{\matheurm{fbc}}
\newcommand{\fin}{\matheurm{fin}}
\newcommand{\mfin}{\matheurm{mfin}}
\newcommand{\mvc}{\matheurm{mvc}}
\newcommand{\Spectra}{\matheurm{Spectra}}

\newcommand{\vc}{\matheurm{vc}}

\newcommand{\R}{\mathbb{R}}
\newcommand{\Z}{\mathbb{Z}}

\newcommand{\bK}{\mathbf{K}}
\newcommand{\bL}{\mathbf{L}}
\newcommand{\bP}{\mathbf{P}}

\newcommand{\mcC}{\mathcal{C}}

\newcommand{\mcF}{\mathcal{F}}
\newcommand{\mcG}{\mathcal{G}}
\newcommand{\mcN}{\mathcal{N}}
\newcommand{\mcO}{\mathcal{O}}

\newcommand{\mcS}{\mathcal{S}}

\newcommand{\G}{{\Gamma}}

\newcommand{\p}{{\partial}}

\DeclareMathOperator{\Aut}{Aut}
\DeclareMathOperator{\CAT}{CAT}
\DeclareMathOperator{\Cl}{Cl}

\DeclareMathOperator{\Cyl}{Cyl}

\DeclareMathOperator{\id}{id}

\DeclareMathOperator{\Int}{Int}

\DeclareMathOperator{\midd}{mid}
\DeclareMathOperator{\Or}{Or}

\DeclareMathOperator{\Torus}{Torus}
\DeclareMathOperator{\UNil}{UNil}
\DeclareMathOperator{\Wh}{Wh}

\newcommand{\TOP}{{\mathrm{TOP}}}

\DeclareMathOperator*{\colim}{colim}

\newcommand{\bdry}{\partial}
\newcommand{\connective}{\langle 1 \rangle}
\newcommand{\eps}{\varepsilon}
\newcommand{\infdec}{{-\infty}}
\newcommand{\iso}{\cong}
\newcommand{\longra}{\longrightarrow}
\newcommand{\ol}[1]{\overline{#1}}
\newcommand{\rel}{\;\mathrm{rel}\;}
\newcommand{\tp}{$\,{}^{\prime}$} 
\newcommand{\ul}[1]{\underline{#1}}
\newcommand{\utilde}{^{\sim}}
\newcommand{\wh}[1]{\widehat{#1}}
\newcommand{\wt}[1]{\widetilde{#1}}

\begin{document}

\begin{abstract}
The problem of equivariant rigidity is the $\Gamma$-homeomorphism classification of $\Gamma$-actions on manifolds with compact quotient and with contractible fixed sets for all finite subgroups of $\Gamma$.
In other words, this is the classification of cocompact $E_\fin \Gamma$-manifolds.

We use surgery theory, algebraic $K$-theory, and the Farrell--Jones Conjecture to give this classification for a family of groups which satisfy the property that the normalizers of nontrivial finite subgroups are themselves finite.
More generally, we study  cocompact proper actions of these groups on contractible manifolds and prove that the $E_\fin$ condition is always satisfied. 
\end{abstract}
\maketitle

\section{Introduction}

\begin{undefn}
Let $\G$ be a discrete group.
Define $\mcS(\Gamma)$ to be the set of equivariant homeomorphism classes $[M,\G]$ of contractible topological manifolds $M$ equipped with an effective cocompact proper $\Gamma$-action.
The \emph{singular set} and \emph{free set} are
\begin{eqnarray*}
M_{sing} &:=&  \{x \in M \mid \exists g\neq 1 : gx=x\}\\
M_{free} &:=&  \{x \in M \mid \forall g\neq 1 : gx \neq x\}.
\end{eqnarray*}
\end{undefn}

When $\G$ is a torsion-free group of isometries of a nonpositively curved Riemannian manifold $M$, with $[M,\G]\in\mcS(\G)$, an inspirational result of
Farrell and Jones \cite{FJ_rigidity} says that $\mcS(\G)$ consists of the single element $[M,\G]$.
(See \cite{BL_CAT0} for the generalization to the $\CAT(0)$ case.) 
In these cases, $M_{sing}$ is empty. 

In this paper we compute $\mcS(\G)$ when $\G$ is a group of isometries of a nonpositively curved Riemannian manifold $M$ where $M_{sing}$ is discrete (that is, the action $\Gamma \curvearrowright M$ is \emph{pseudo-free}).
One may speculate that  $\mcS(\G)$ consists of just one element.
However, for many years experts have suspected that Cappell's $\UNil$-groups (see Quinn \cite{Quinn_ICM}) obstruct this hope.  
We fully calculate $\mcS(\G)$.
 
\begin{thm}[Main Result]\label{thm:main}
Let $\Gamma$ be a group satisfying Hypotheses (ABC) or more generally (ABC\tp) below, with $\G$ of virtual cohomological dimension $n \geqslant 5$.
Write $\eps:= (-1)^n$.
There is a bijection
\[\begin{tikzcd}
\displaystyle\bigoplus_{(\midd)(\Gamma)} \UNil_{n+\eps}(\Z; \Z, \Z) \rar{\bdry}[swap]{\approx} & \mcS(\Gamma).
\end{tikzcd}\]
The set of conjugacy classes of maximal infinite dihedral subgroups of~$\Gamma$ is $(\midd)(\Gamma)$.
\end{thm}

These unitary nilpotent groups have been computed  by Banagl, Connolly, Davis, Ko\'zniewski, and Ranicki \cite{CK_nilgps, CD, CR, BR}.
These abelian groups are
\begin{equation*}
\UNil_m(\Z;\Z,\Z) ~\cong~ \begin{cases}
0 & \text{if } m \equiv 0 \pmod{4}\\
0 & \text{if } m \equiv 1 \pmod{4}\\
(\Z/2\Z)^\infty & \text{if } m \equiv 2 \pmod{4}\\
 (\Z/2\Z \oplus \Z/4\Z)^\infty & \text{if } m \equiv 3 \pmod{4}.
\end{cases}
\end{equation*}
Thus $\mcS(\G)$ is a singleton if $n\equiv 0,1 \pmod{4}$ or if $\G$ has no elements of order 2.

\begin{hyp}
Let $\G$ be a discrete group.
\begin{enumerate}
\item[(A)]
There exists a torsion-free subgroup $\Gamma_0$ of $\Gamma$ with finite index.

\item[(B)]
The normalizer $N_\Gamma(H)$ is finite for each nontrivial finite subgroup $H$ of $\Gamma$.

\item[(C)]
There is a contractible Riemannian manifold $X$ of nonpositive sectional curvature with an effective cocompact proper action $\Gamma \curvearrowright X$ by isometries.
\end{enumerate}

\noindent A more general but less geometric condition than (C) is (C\tp), consisting of:
\begin{enumerate}
\item[(C\tp i)]
There exists an element $[X, \Gamma]\in \mcS(\Gamma)$ for which the quotient space $X_{free}/\Gamma$ has the homotopy type of a finite CW complex.

\item[(C\tp ii)]
Each infinite dihedral subgroup of $\Gamma$ lies in a unique maximal infinite dihedral subgroup.

\item[(C\tp iii)]
$\Gamma$ satisfies the Farrell--Jones Conjecture in $K$-theory and $L$-theory.
\end{enumerate}
\end{hyp}

\begin{rem}
For comparative purposes, we make some comments here.
\begin{itemize}
\item
Proposition~\ref{prop:onepoint} shows that if $\G$ satisfies Hypothesis~(A) and 
$[M,\G] \in\mcS(\G)$, then $\G$ satisfies Hypothesis~(B)  if and only if $M_{sing}$ is a discrete
set.

\item
Suppose $\Gamma$ satisfies Hypotheses~(AB).
Corollary~\ref{cor:phi_bijects} shows that any two $\G$-manifolds  in $\mcS(\G)$ model $E_\fin\G$ and hence are $\G$-homotopy equivalent.

\item
(ABC) $\Longrightarrow$ (ABC\tp):
Suppose $\G$ satisfies Hypotheses (AB) and $X$ satisfies Hypothesis~(C).
Then clearly $[X,\G] \in \mcS(\G)$ and $X_{free}/\G$ has the homotopy type of the finite CW complex $(X - N(X_{sing}))/\G$.
Thus Hypothesis~(C\tp i) holds.
Let $\Delta$ be a infinite dihedral subgroup of $\Gamma$.
Then $\Delta$ stabilizes a unique geodesic $\gamma$, namely the geodesic which contains the fixed points of the involutions of $\Delta$, noting the fixed set of each involution in $\Delta$ is a point.
The stabilizer of $\gamma$ is the unique maximal infinite dihedral subgroup of $\G$ containing $\Delta$.
Thus Hypothesis~(C\tp ii) holds.
Finally Hypothesis~(C\tp iii) holds by \cite{BL_CAT0} and \cite{Wegner}.

\item
In the proof of Theorem~\ref{thm:main} we will only refer to Hypotheses (ABC\tp).  

\item
Theorem~\ref{thm:main} was proven for the case $(X,\G)=(\R^n,\Z^n \rtimes_{-1} C_2)$ in \cite{CDK}.
\end{itemize}
\end{rem}

We explain how $\mcS(\G)$ is the isovariant structure set $\mcS^{iso}(X,\G)$ of any $[X,\G]$.

\begin{undefn}
Let $(X,\G)$ be a cocompact $\G$-manifold.
A \emph{structure} on $(X,\G)$ is a $\G$-equivariant homotopy equivalence $f:M\to X$ of proper cocompact $\G$-manifolds.
If $f$ is an isovariant homotopy equivalence, we call it an \emph{isovariant structure}.
Two such structures (resp.~isovariant structures), $(M,f,\G)$ and $(M',f',\G)$ are \emph{equivalent} if there is an equivariant homeomorphism $h:M\to M'$
and a $\G$-equivariant homotopy (resp. $\G$-isovariant homotopy) from $f'\circ h$ to $f$.
Define the \emph{structure set} $\mcS(X,\G)$ (resp.~\emph{isovariant structure set} $\mcS^{iso}(X,\G)$) as the set of equivalence classes of structures
(resp.~isovariant structures) on $(X,\G)$.
\end{undefn}

\begin{thm}\label{thm:equi-iso}
Suppose $\G$ satisfies Hypotheses~(AB).
Assume there is an element $[X^n,\G] \in \mcS(\G)$ such that $n \geqslant 4$.
Then the forgetful maps are bijections:
\[
\begin{tikzcd}\mcS^{iso}(X, \G) \rar{\psi} & \mcS(X, \G) \rar{\phi} & \mcS(\G)\end{tikzcd} ;~
[M,f,\G] ~\longmapsto~ [M,f,\G] ~\longmapsto~ [M,\G].
\]
\end{thm}

\begin{cor}\label{cor:2torsionfree}
Suppose $\Gamma$ satisfy Hypotheses~(ABC\tp) with the dimension of $X$ at least five.
Assume $\G$ has no elements of order two.
Then any contractible manifold $M$ equipped with an effective cocompact proper $\G$-action has the $\G$-homotopy type of $X$, and  any $\Gamma$-homotopy equivalence $f: M \to X$ is $\Gamma$-homotopic to a $\G$-homeomorphism.
\end{cor}

\begin{proof}
This follows from Theorems~\ref{thm:main} and~\ref{thm:equi-iso},  since $(\midd)(\Gamma)$ is empty.
\end{proof}

When $\G$ is a torsion-free group, Corollary~\ref{cor:2torsionfree} follows from \cite[Theorem~B]{BL_CAT0}.

It is interesting to note that our theorems have no \emph{simple} homotopy equivalence requirement; equivariant homotopy equivalence up to equivariant
homeomorphism is enough. The reason is the isomorphism of the assembly map in the next result.

\begin{prop}[cf.~Theorem~\ref{thm:sumWh}]\label{prop:sumWh}
Suppose $\G$ satisfies Hypothesis~(B). Then
\[
H^\G_*(E_\vc\G, E_\fin\G; \ul{\bK}_\Z) ~=~ 0.
\]
\end{prop}

This vanishing result can be interpreted geometrically, as follows.
Recall from \cite[Section~1]{CDK} that a \emph{cocompact manifold model} for $E_\fin\G$ is a $\G$-space with the equivariant homotopy type of a $\G$-CW complex such that, for all subgroups $H$ of $\G$, the $H$-fixed set is a contractible manifold if $H$ is finite and is empty if $H$ is infinite.
A $\G$-space $W$ is \emph{locally flat} if $W^H \subset W^K$ implies $W^H$ is a locally flat submanifold of $W^K$.
For our $h$-cobordism rigidity, we do not assume that $\G$ acts locally linearly.

\begin{thm}\label{thm:hcobordism_rigidity}
Suppose $\G$ satisfies Hypotheses~(AB) and the Farrell--Jones Conjecture in  $K$-theory.
Assume $M$ and $M'$ are cocompact manifold models for $E_\fin \G$ of dimension at least six.
If there exists a locally flat isovariant $h$-cobordism, $W$,  between $M$ and $M'$, then $W$ is a product cobordism and $M$ and $M'$ are equivariantly homeomorphic.
\end{thm}  

Here we outline the logical structure of arguments in the rest of the paper.
In Sections~\ref{sec:onepoint} and \ref{sec:eq_to_iso} we establish the  geometric topology needed to prove Theorem~\ref{thm:equi-iso}.
In Section~\ref{sec:repulsion}, we use the end theory of high-dimensional manifolds to reduce the isovariant structure set to a classical structure set of a
certain compact pair, $(\ol{X},\bdry\ol{X})$.
Therein we use an algebraic result (Corollary~\ref{cor:sumWh}) which depends on Proposition~\ref{prop:sumWh}.
Later, we further reduce to Ranicki's algebraic structure groups.
In Section~\ref{sec:lowering}, we systematically remove all of the  $K$-theory decorations and pass to $\mcS^{-\infty}$.
In Section~\ref{sec:structures}, we calculate these structure groups using the Farrell--Jones Conjecture and the other axioms for $\G$, thus proving Theorem~\ref{thm:main}.
In Section \ref{sec:eq_h}, we deduce Theorem \ref{thm:hcobordism_rigidity} from Corollary \ref{cor:2torsionfree} and general results of Quinn.  In Section \ref{sec:examples}, we give a variety of examples of groups which satisfy the hypotheses of Theorem \ref{thm:main}, and hence groups where the question of equivariant rigidity is completely answered.

\subsection*{Acknowledgments}
The authors thank both the referee for suggesting to use hyperbolization to find examples and also Mike Davis for discussions on this topic.

\section{$E_\fin \G$-manifolds}\label{sec:onepoint}

Recall that for a discrete group $\G$, a \emph{model for $E_{\fin} \Gamma$} is a $\Gamma$-space $M$ which is $\Gamma$-homotopy equivalent to a $\Gamma$-CW complex and, for all subgroups $H$ of $\Gamma$,
\[
M^H \text{ is }
\begin{cases}
\text{contractible} & \text{if } H \text{ is finite}\\
\text{empty} & \text{otherwise}.
\end{cases}
\]
Given any $\Gamma$-CW complex $X$ with finite isotropy groups, there is an equivariant map $X \to E_{\fin}\Gamma$, unique up to equivariant homotopy.  It follows that any two models are $\Gamma$-homotopy equivalent.  Furthermore, a model $E_{\fin}\Gamma$ exists for any group $\Gamma$.  

We show that if a group $\G$ satisfies Hypotheses (AB) and acts effectively, cocompactly, and properly on a contractible manifold, then the manifold models $E_{\fin}\Gamma$.

The following lemma is not logically necessary, but nonetheless provides motivation.

\begin{lem}\label{lem:AC_Efin}
Let $\G$ be a virtually torsion-free group with an effective cocompact proper action by isometries on a contractible Riemannian manifold $X$ of nonpositive sectional curvature (that is, Hypotheses (AC) hold).
Then $X$ models $E_\fin\G$.   
\end{lem}

\begin{proof}
Let $\G_0$ be a finite index normal subgroup of $\G$ with quotient group $G$.  The manifold $X$ is complete since $X/\G_0$ is compact, hence complete. 
For an infinite subgroup $H$ of $\G$, the fixed set $X^H$ is empty by properness.  
For a finite subgroup $H$  of $\G$, the fixed set $X^H$ is nonempty by the Cartan Fixed Point Theorem  \cite[Theorem 13.5]
{Helgason}, which states that a compact group of isometries of a complete nonpositively curved manifold has a fixed point.   It is easy to show that the fixed set of group of isometries is geodesically convex, hence contractible if nonempty.
Finally, since the finite group $G$ acts smoothly on the closed manifold $X/\G_0$, this is a $G$-CW complex \cite{Illman_G}, and so, by lifting cells, $X$ is a $\G$-CW complex.  
\end{proof}

Now it remains to show that if $\G$ satisfies Hypotheses~(AB) and if $M$ is a contractible $\G$-manifold, then the fixed point set of each finite subgroup of $\G$ is a point and $M$ has the $\G$-homotopy type of a $\G$-CW complex; hence $M$ will model $E_{\fin}\Gamma$.
We start with a group-theoretic lemma.  

\begin{lem}\label{lem:ft}
Suppose a finite group $G$ acts on a set $Y$ so that for all prime-order subgroups $P$ of $G$, the fixed set $Y^P$ is a singleton. Then $Y^G$ is a
singleton.
\end{lem}

\begin{proof}
Let $\mcC$ be the collection of nontrivial subgroups of $G$ whose fixed set is a singleton.
Note that $\mcC$ satisfies the \emph{normal extension property}: if $N \triangleleft H$ with $N \in \mcC$, then $H \in \mcC$.
Indeed $Y^H \subset Y^N$ since $N \subset H$.
On the other hand, since $N\in \mcC$, we see the singleton $Y^N$ is an $H$-invariant  by normality, so $Y^N \subset Y^H$.
So $Y^H$ is a point and $H\in \mcC$.
In particular $\mcC$ contains all cyclic subgroups of $G$.

We next claim that, if $a$ and $b$ are elements of order 2 in $G$, then $X^{\langle a\rangle}=X^{\langle b\rangle}=X^{\langle a,b \rangle}$.
For the cyclic group $\langle ab\rangle$ is in $\mcC$, and $\langle ab\rangle\triangleleft \langle a, b \rangle$. So $\langle a, b \rangle\in \mcC$. Therefore $X^{\langle a\rangle}$ and $X^{\langle b\rangle}$ are the same point $X^{\langle a,b \rangle}$.
It follows that this point is the fixed set of the normal subgroup $N$ generated by all elements of order 2 in $G$.
This subgroup is nontrivial if $G$ has even order. So by the normal extension property, $Y^G$ is a singleton if $G$ has even order.

Finally suppose $G$ has odd order. It that order is a prime, then $G$ is in $\mcC$. If not then by the Feit--Thompson Theorem \cite{FeitThompson}, there is a
nontrivial proper normal subgroup and we proceed inductively to conclude that
$G\in \mcC$.
\end{proof}

Recall that a group action is \emph{pseudo-free} if the singular set is discrete.
Our next result shows that the finite normalizer condition is equivalent to pseudo-free. 

\begin{prop}\label{prop:onepoint}
Suppose $\G$ satisfies Hypothesis~(A).
Assume $M$ is a contractible manifold equipped with an effective cocompact proper $\G$-action. The action $\G \curvearrowright M$ is pseudo-free if and only if Hypothesis~(B) holds.
Furthermore, if either of these equivalent conditions holds, then for a subgroup $H$ of $\G$, the fixed point set $M^H$ is empty if $H$ is infinite  and a point if $H$ is nontrivial and finite.
\end{prop}

\begin{proof}
First suppose $\G \curvearrowright M$ is pseudo-free.
Fix a nontrivial finite subgroup $H$ of $\G$.
Let $P$ be a prime-order subgroup of $H$.
By \cite[Lemma 2.2(2)]{CDK}, $M^P$ is nonempty and connected.
Then, since $M^P$ is a subspace of the discrete space $M_{sing}$, $M^P$ must be a single point.
Thus, by Lemma~\ref{lem:ft}, $M^H$ is a single point.
Therefore, since the restriction $N_\G(H) \curvearrowright  M^H$ is proper, $N_\G(H)$ must be finite.

Now suppose, in addition to the existence of a torsion-free subgroup $\G_0 \triangleleft \G$ of finite index, that $N_\G(H)$ is a finite group for each nontrivial finite subgroup $H$ of $\G$.
Let $P$ be a prime-order subgroup of $\G$.   
 The covering map $p : M \to M/\G_0$ is $P$-equivariant. By \cite[Lemma 2.3(3)]{CDK}, $M^P \to p(M^P)$ is a regular
$\G_0^P$-cover where $\G_0^P = \{ \gamma \in \G_0 \mid \gamma h \gamma^{-1} = h\in \G_0 \text{ for all }h \in P\}$. But $\G_0^P \subset \G_0 \cap N_\G(P) = 1$,
so $M^P \to p(M^P)$ is a homeomorphism. By \cite[Lemma 2.3(6)]{CDK}, $p(M^P)$ is compact, and hence so is $M^P$. By \cite[Lemma 2.2(3)]{CDK}, $M^P$ is a
point since it is both compact and the fixed set of a $p$-group action on a contractible manifold. Thus, by Lemma~\ref{lem:ft}, $M^H$ is a point for each
nontrivial finite subgroup $H$ of $\G$.

Note that since the action is proper, $M^H$ is empty for infinite subgroups $H$ and that we showed above that $M^H$ is a point when $H$ is nontrivial and finite.   
\end{proof}
 
\begin{rem}
The resort to the Feit--Thompson theorem in Lemma~\ref{lem:ft} is striking. But for our purposes   (in the proof of Proposition~\ref{prop:onepoint}) it is not actually necessary.
For if $G$ is an odd order subgroup of $\G$, the argument in Lemma~\ref{lem:ft} shows that each nontrivial Sylow $p$-subgroup $P$ of $G$ fixes exactly one point
(use the normal extension property and the nilpotency of $P$). Since $P$ acts freely on the homology sphere $X-X^P$, it follows that $P$ has periodic
cohomology and is therefore cyclic, since $p$ is odd. As is well known, this implies that $G$ is metacyclic. So again by the normal extension property,
since $G$ has a normal cyclic subgroup, we see $G\in \mcC$ and $X^G$ is a singleton.
\end{rem}

Here is a general tameness result pertaining to all of the actions that we consider.

\begin{prop}\label{prop:GCWstructure}
Let $\G$ be a discrete group.
Any topological manifold $M$ equipped with a cocompact pseudo-free $\G$-action has the $\G$-homotopy type of a $\G$-CW complex.
\end{prop}

\begin{proof}
Since $M_{sing}$ is a discrete subset of the topological manifold $M$, it is forward-tame (tame in the sense of Quinn \cite{Quinn_HSS}).
Then, by \cite[Propositions 2.6 and 3.6]{Quinn_HSS}, $M_{sing}/\G$ is forward-tame in $M/\G$. 
So, by \cite[Proposition 2.14]{Quinn_HSS}, the complement $M_{free}/\G = M/\G - M_{sing}/\G$ is reverse-tame (tame in the sense of Siebenmann \cite{Siebenmann_thesis}).
 
Then by Siebenmann \cite{Siebenmann_thesis}, if $\dim M \geqslant 5$, the ends of $ M_{free}/\G \times S^1$ have collar neighborhoods.   We will go ahead and cross with the 5-torus, and thereby guarantee both that the manifold is high-dimensional  and the ends are collared.  
Hence there is a compact $\partial$-manifold $W$ such that
\[
M_{free}/\G \times T^5 ~=~ (M_{free} \times T^5)/\G ~=~ W \cup_{\partial W} (\partial W \times [0,1)).
\]

Write $\wt W$ and $ \wt{ \partial W}$ for the pullback $\G \times \Z^5$-covers of $W$ and $\partial W$ of the cover $M_{free} \times \R^5 \to M_{free}/\G \times T^5$.
Observe $\G$ acts properly on the set $\Pi_0 \wt{\partial W}$ of connected components of $\wt{\partial W}$.
Then, since $\G \curvearrowright M$ is pseudo-free, we can define a $\G$-map
\begin{equation}\label{e_map}
e: \wt{\partial W} \longrightarrow M_{sing} ~;~ s \longmapsto M^{\G_{C(s)}},
\end{equation}
where $C(s) \in \Pi_0 \wt{\partial W}$ is the connected component containing a point $s$.   Consider the equivalence relation on $M \times T^5$ given by $(s,z) \sim (s,z')$ if $s \in M_{sing}$.
This is preserved by the $\G$-action.
Write $M_1 = (M \times T^5)/\! \sim$ for the quotient.
As $\G$-spaces,
\[
M_1 ~=~ \wt{W} \cup_{ \wt{\partial W}} \Cyl(e).
\]
The mapping cylinder $\Cyl(e) := \wt{\partial W} \times [0,1] \cup_{e \times 1} M_{sing}$ has the quotient topology.

Since $W$ is a topological $\bdry$-manifold, by Kirby--Siebenmann \cite[p123, Theorem III:4.1.3]{KS}, there is a proper homotopy equivalence $(f,\bdry f): (W,\bdry W) \to (L,K)$ for some CW pair $(L,K)$.
Also $M_{sing}$ is a 0-dimensional $\G$-CW complex.
Then $M_1$ has the $\G$-homotopy type of a $\G$-CW complex, via the $\G$-homotopy equivalence
\[
\wt{f} \cup \Cyl\left(\wt{\bdry f}\right) ~:~ M_1 ~\longrightarrow~ \wt{L} \cup_{\wt{K}} \Cyl\left(e\circ \wt{(\bdry f)^{-1}}\right). 
\]
Here $\wt{(\bdry f)^{-1}}$ denotes the $\G$-cover of a homotopy-inverse for $\bdry f$.

Finally, choose a point $* \in T^5$ and consider the $\G$-inclusion and corresponding $\G$-retraction
\[
i: M \longrightarrow M_1 ~; y \longmapsto [y,*]
\quad\text{and}\quad
r: M_1 \longrightarrow M ~;~ [y,z] \longmapsto y.
\]
The mapping torus $\Torus(i\circ r)$ has the $\G$-homotopy type of a $\G$-CW complex.
By cyclic permutation of composition factors, $M \simeq M \times \R = \overline{\Torus(r\circ i)} \simeq \overline{\Torus(i \circ r)}$ as $\G$-spaces, where ${}^-$ means the infinite cyclic cover.
This completes the proof.
\end{proof}

\begin{cor}\label{cor:phi_bijects}
Suppose $\G$ satisfies Hypotheses~(AB).
Any contractible manifold $M$ equipped with an effective cocompact proper $\G$-action is a model for $E_\fin\G$.
Hence the forgetful map $\phi: \mcS(X,\G) \to \mcS(\G)$ of Theorem~\ref{thm:equi-iso} is bijective. 
\end{cor}

\begin{proof}
By Proposition~\ref{prop:onepoint}, the $H$-fixed sets of $M$  are single points if $H$ is a nontrivial finite subgroup of $\G$,
otherwise the $H$-fixed sets are empty if $H$ is infinite.
Hence $M$  is a pseudo-free $\G$-manifold.
Then, by Proposition~\ref{prop:GCWstructure}, $M$  has the $\G$-homotopy type of a $\G$-CW complex.
Therefore, since $M$ is contractible, it models $E_\fin\G$.
The remainder follows from universal properties of classifying spaces.
\end{proof}

\section{From Equivariance to Isovariance} \label{sec:eq_to_iso}

We recall our earlier result about improving equivariant maps:
\begin{thm}[{\cite[Theorem~A.3]{CDK}}]\label{thm:CDK_AppA}
Let $G$ be a finite group.
Let $A^n$ and $B^n$ be compact  pseudo-free $G$-manifolds.
Assume $n\geqslant 4$.
Let $f:A\to B$ be a $G$-map such that the restriction $f|_{A_{sing}}: A_{sing}\to B_{sing}$ is bijective.
\begin{enumerate}
\item
If $f$ is $1$-connected, then $f$ is $G$-homotopic to an isovariant map.

\item
If $f= f_0$ and $f_0,f_1: A\to B$ are $G$-isovariant, and $2$-connected, and   $G$-homotopic, then $f_0$ is $G$-isovariantly homotopic to $f_1$.
\end{enumerate}
\end{thm}

\begin{prop}\label{prop:psi}
The map $\psi:\mcS^{iso}(X,\G)\to \mcS(X,\G)$ of Theorem~\ref{thm:equi-iso} is bijective.
\end{prop}

\begin{proof}
Let $[M,f,\G]\in \mcS(X,\G)$.
Since $f: M \to X$ is a $\G$-homotopy equivalence of $\G$-spaces, for every subgroup $H$ of $\G$, the restrictions $f^H: M^H \to X^H$ are homotopy
equivalences.
Now, by Proposition~\ref{prop:onepoint}, for each nontrivial finite subgroup $H$ of $\G$, $M^H$ and $X^H$ are single points.
Thus $f_{sing}: M_{sing} \to X_{sing}$ is a bijection.

By hypothesis, there exists a normal torsion-free subgroup $\G_0 \triangleleft \G$ with $G := \G/\G_0$ finite.
Since the restricted actions $\G_0 \curvearrowright M$ and $\G_0 \curvearrowright X$ are free, proper, and cocompact, the quotients $M/\G_0$ and $X/\G_0$ are compact $G$-manifolds.
By Theorem~\ref{thm:CDK_AppA} applied to both the $G$-homotopy equivalence $f/\G_0: M/\G_0 \to X/\G_0$ and any choice of $G$-homotopy-inverse, $f/\G_0$ is
$G$-homotopic to a $G$-isovariant homotopy equivalence $g: M/\G_0 \to X/\G_0$, unique up to $G$-isovariant homotopy.

By the Covering Homotopy Property applied to the cover $M \to M/\G_0$, the $G$-homotopy from $f/\G_0$ to $g$ is covered
by a unique homotopy $F: M \times I \to X$ from $f: M \to X$, to some $\hat{g}: M \to X$ covering $g$.
By uniqueness of path lifting, it follows that $F$ is $\G$-equivariant.
Then, since $\hat{g}$ is $\G$-equivariant and $g$ is $G$-isovariant,  an elementary diagram
chase shows that $\hat{g}$ is a $\G$-isovariant map.
Since $\hat{g}$ is a $\G$-equivariant map covering a $G$-isovariant homotopy equivalence $g$, by similar reasoning it follows that $\hat{g}$ is a
$\G$-isovariant homotopy equivalence, unique up to $\G$-isovariant homotopy.
Then $[M,\hat{g},\G] \in \mcS^{iso}(X,\G)$ and $\psi[M,\hat{g},\G] = [M,f,\G]$.
Thus $\psi$ is surjective.
Furthermore, since $g$ is unique up to $G$-isovariant homotopy and $\hat{g}$ is unique up to $\G$-isovariant homotopy, $\psi$ is injective.
\end{proof}

\begin{proof}[Proof of Theorem~\ref{thm:equi-iso}]
Immediate from Corollary~\ref{cor:phi_bijects} and Proposition~\ref{prop:psi}.
 \end{proof}

\section{Reduction to classical surgery theory}\label{sec:repulsion}

In this section we show that each cocompact proper contractible $\G$-manifold $[M,\G]$ in $\mcS(\G)$ is well-behaved in a neighborhood of the discrete set $M_{sing}$. We use this to interpret $\mcS^{iso}(X,\G)$ as a structure set of a compact manifold-with-boundary, which we call a \emph{compact $\bdry$-manifold}.
We begin with a quick fact about uniqueness.

\begin{lem}\label{lem:sphere}
Let $A$ be a manifold of dimension greater than four, and let $a \in A$.
Let $C$ be a compact neighborhood of $a$ in $A$ so that $C-\{a\}$ is homeomorphic to $B \times [0,\infty)$ where $B$ is a manifold.
Then $C$ is homeomorphic to a disc.
\end{lem}

Denote the \emph{closed cone} on a space $S$ by $cS := S \times [0,1] / S \times \{0\}$.

\begin{proof}
There is a homeomorphism $\varphi: B \times [0,\infty) \to C-\{a\}$.
By uniqueness of the one-point compactification of a space,  there are basepoint-preserving homeomorphisms:
\[
C ~=~ \Cl_A \varphi(B \times [0,\infty))
~\cong~ (B \times [0,\infty))^+
~\cong~ cB.
\]
Observe that $C$ is a $\bdry$-manifold with $\bdry C = B$, hence $B$ is compact. 
Note also that $B$ is simply-connected, since $\pi_1B  \cong \pi_1(C-\{a\}) \cong \pi_1C \cong \pi_1 cB  \cong 1$ where the second isomorphism follows from the Seifert--van Kampen Theorem applied to the decomposition $C = \Int D \cup (C - \{a\})$ with  $D$  a closed disc neighborhood of $a$ in the manifold $\Int{C}$.
Excision shows that $\bdry D \hookrightarrow (C-\Int{D})$ induces an isomorphism on homology.
Thus $(C-\Int{D};\bdry D, B)$ is an $h$-cobordism, and the $h$-cobordism theorem \cite{Smale_Poincare,Freedman_Poincare} shows that is $C - \Int D$ is homeomorphic to the product $\bdry D \times I$.
Hence $B \cong \bdry D$, a sphere and $C \cong D$, a disc.
\end{proof}

An action of a group $\G$ on a space $X$ is \emph{locally conelike} if every orbit $\G x$ has a $\G$-neighborhood that is $\G$-homeomorphic to $\G \times_{\G_x} cS(x)$ for some $\G_x$-space $S(x)$.

\begin{lem}\label{lem:mwb}
Suppose $\G$ satisfies Hypotheses~(ABC\tp) holds with $\dim X \geqslant 5$.
Let $[M,f,\G] \in \mcS^{iso}(X,\G)$.
There is a compact  $\bdry$-manifold $\ol{M}$ with interior $M_{free}/\Gamma$.
Furthermore, the action $\G \curvearrowright M$ is locally conelike.
\end{lem}

\begin{proof}
Since $M_{sing}/\G$ is compact and since   $M_{sing}$ is discrete by Proposition~\ref{prop:onepoint}, $M_{sing}/\G$ is a finite set.  
Then the manifold $M_0 := M_{free}/\G$ has finitely many ends, one for each orbit $ \G p_1, \ldots,  \G p_m$ of $M_{sing}$.
Write $\Gamma_j := \Gamma_{p_j}$ for the isotropy groups.

Since the ends of $M_{free}$ are tame, and each end of $M_0$ is finitely covered by an end of $M_{free}$, Propositions~2.6 and 3.6 of \cite{Quinn_HSS} shows the ends of
$M_0$ are tame.
If $M_0$ were smooth, then Siebenmann's thesis \cite{Siebenmann_thesis} gives a CW structure on $M_0$ which is a union of subcomplexes $M_0 = K \cup E_1 \cup \cdots \cup E_m$ with $K$ a finite complex and each $E_i$ a connected, finitely dominated complex with $\pi_1 E_j = \G_j$.
For $M_0$ a topological manifold, it follows from \cite[p123, Theorem III:4.1.3]{KS} that there is a proper homotopy equivalence $M_0 \simeq K \cup E_1 \cup \cdots \cup E_m$ with $K$ and $E_j$ as above.
Let $\Wh_0(\G_j)$ denote the reduced projective class group $\wt K_0(\Z\G_j)$ and $\sigma(E_j) \in \Wh_0(\G_j)$  the Wall finiteness obstruction. 

Consider inclusion-induced map of projective class groups:
\[
A_0 : \begin{tikzcd}\displaystyle\prod_{j=1}^m \Wh_0(\G_j) \rar & \Wh_0(\G).\end{tikzcd}
\]  
Since $K$ is a finite complex, the sum theorem of \cite{Siebenmann_thesis} shows 
\[
A_0(\sigma(E_1), \dots, \sigma(E_m)) = \sigma(M_0).
\]
On the one hand,  $X_0 := X_{free}/\G$ has the homotopy type of a finite CW complex by Hypothesis~(C\tp i), and $f_{free}/\G: M_0 \to X_0$ is a homotopy equivalence,  
so $ \sigma(M_0)= \sigma(X_0) = 0$.
On the other hand, $A_0$ is injective by Corollary~\ref{cor:sumWh} below.  Hence $\sigma(E_i) = 0$ for each $i$.  
Therefore, by Siebenmann's theorem (\cite{Siebenmann_thesis}, \cite[p.~214]{FQ}), there is a compact $\bdry$-manifold $\ol{M}$ with interior $M_0$.

Furthermore, $M-M_{sing} = \wt{M_0}$ is $\G$-homeomorphic to the universal cover $\wt{\ol{M}}-\bdry\wt{\ol{M}}$.
Consider the closure $C := \Cl_M\left(\bdry\wt{\ol{M}} \times [0,\infty)\right)$, which is a $\G$-cocompact neighborhood of $M_{sing}$ in $M$.
By Lemma~\ref{lem:sphere}, each component of $ C$ is a disc.
Then
\[
C ~\cong~ \bigsqcup_{j=1}^m \G \times_{\G_j} cS^{n-1}(p_j).
\]
Therefore the action of $\G$ on $M$ is locally conelike.
\end{proof}

By Lemma~\ref{lem:mwb}, we may choose a compact $\bdry$-manifold $\ol{X}$ with interior $X_{free}/\G$.
Note $\pi_1(\ol{X}) \iso \G$.
Enumerate the connected components of the boundary:
\[
\bdry\ol{X} ~=~ \bigsqcup_{j=1}^m \bdry_j\ol{X}.
\]
Observe each $\bdry_j\ol{X}$ has universal cover homeomorphic to $S^{n-1}$ and $\pi_1(\bdry_j\ol{X}) \iso \G_j$.

\begin{undefn}
A \emph{structure} on $(\ol{X}, \bdry\ol{X})$ is a pair $(\ol{M},f)$ where $\ol{M}$ is a compact topological $\bdry$-manifold and $f:(\ol{M},\bdry\ol{M})\to
(\ol{X},\bdry\ol{X})$ is a homotopy equivalence of pairs.
Two structures $(\ol{M},f)$ and $(\ol{M}{}',f')$ are equivalent if there are an $h$-cobordism of pairs, $(W,\bdry_0 W)$ from $(\ol{M},\bdry\ol{M})$ to $(\ol{M}{}', \bdry\ol{M}{}')$ and an extension $F$ of $f \sqcup f'$. The \emph{structure set} $\mcS^{h}_\TOP(\ol{X},\bdry\ol{X})$ is the set of equivalence classes of structures on $(\ol{X},\bdry\ol{X})$.
\end{undefn}

\begin{lem}\label{lem:compactification}
Suppose $\G$ satisfies Hypotheses~(ABC\tp) and $n \geqslant 5$.
There is a bijection
\[
\Phi : \begin{tikzcd}\mcS^{h}_\TOP(\ol{X},\bdry\ol{X}) \rar & \mcS^{iso}(X,\G)\end{tikzcd} ;~ [\ol{X},\id] \longmapsto [X,\id,\G].
\]
\end{lem}

\begin{proof}
We define $\Phi$ on representatives as follows.
Let $(\ol{f}, \bdry\ol{f}): (\ol{M}, \bdry\ol{M}) \to (\ol{X}, \bdry\ol{X})$ be a homotopy equivalence of pairs.
Consider the $\G$-space defined by coning off:
\[
M ~:=~ (\ol{M})\utilde \,\cup\, \bigsqcup_{j=1}^m \Gamma\times_{\Gamma_j}c((\partial_j \overline{ M})^{\sim})
\]
Here ${}\utilde$ denotes the universal cover.
Since by Lemma~\ref{lem:mwb} each $\bdry_j\ol{M}$ has universal cover $S^{n-1}$,  we obtain that $M$ is a topological manifold.
The map $(\ol{f})\utilde: (\ol{M})\utilde \to (\ol{X})\utilde$ extends to a $\G$-isovariant map $f: M \to X$.
Since $\ol{f}$ and $\bdry\ol{f}$ are homotopy equivalences and $f_{sing}: M_{sing} \to X_{sing}$ is a homeomorphism, it follows that $f: M \to X$ is an $\G$-isovariant homotopy equivalence.
We define
\[
\Phi(\ol{M},\ol{f}) ~:=~ [M,f,\G] ~\in~ \mcS^{iso}(X,\G).
\]

Next, we show that $\Phi$ is constant on equivalence classes.
Let $[\ol{M'}, \ol{f'}] = [\ol{M}, \ol{f}]$ in $\mcS^{h}_\TOP(\ol{X},\bdry\ol{X})$.
By Corollary~\ref{cor:sumWh} with $q=1$, and the fact that $h$-cobordisms are determined up to homeomorphism by their torsion, observe $\ol{M'} = \ol{M}
\cup_{\bdry\ol{M}} W$ for some $h$-cobordism $(W;\bdry\ol{M},\bdry\ol{M'})$, and $\ol{f'} \simeq \ol{f} \cup_{\bdry\ol{f}} F$ for some homotopy equivalence
$F: W \to \bdry\ol{X}$.
By lifting to a $\G$-homotopy on universal covers, it follows that $\Phi(\ol{M'},\ol{f'}) = \Phi(\ol{M} \cup W, \ol{f} \cup F)$.
By an Eilenberg swindle, which uses realization of $h$-cobordisms and triviality of $s$-cobordisms, there is a homeomorphism $W- \bdry\ol{M'} \to
\bdry\ol{M}\ \times [0,1)$ relative to $\bdry\ol{M}$.
This extends to homeomorphism relative to $\bdry\ol{M}$:
\[
W \cup c(\bdry\ol{M'}) = (W - \bdry\ol{M'})^+ ~\longrightarrow~ (\bdry\ol{M} \times [0,1))^+ = c(\bdry\ol{M}).
\]
Here ${}^+$ denotes one-point compactification.
Then we obtain a homeomorphism $M' \to M$ relative to $(\ol{M})\utilde$.
So there is a $\G$-homeomorphism $\phi: M' \to M$ such that $f' \simeq f \circ \phi$.
Hence $[M',f',\G]=[M,f,\G]$.
Therefore $\Phi$ is defined on $\mcS^{h}_\TOP(\ol{X},\bdry\ol{X})$.

We now show that $\Phi$ has a two-sided inverse:
\[
\Psi ~:~ \mcS^{iso}(X, \G) ~\longrightarrow~  \mcS^h_\TOP(\ol{X},\bdry\ol{X}).
\]
Let $[M,f,\G] \in \mcS^{iso}(X, \G)$. Consider the proper homotopy equivalence
\[
f_0 := f_{free}/\G ~:~ M_0 := M_{free}/\G ~\longrightarrow~ X_0 := X_{free}/\G.
\]

By Lemma~\ref{lem:mwb}, there is a compact $\bdry$-manifold $\ol{M}$ with interior $M_0$.
Using a collar for $\bdry\ol{M}$ in $\ol{M}$, after a small proper homotopy of $f_0$, we may assume that $f_0$ extends to a homotopy equivalence $\ol{f}:
(\ol{M}, \p\ol{M})\to (\ol{X},\p\ol{X})$ of pairs.
Define
\[
\Psi(M,f,\G) ~:=~ [\ol{M}, \ol{f}] ~\in~ \mcS_\TOP^{h}(\ol{X}, \bdry\ol{X}).
\]
Observe a different choice for the Siebenmann completion of the ends of $M_0$ would yield a pair $[\ol{M'}, \ol{f'}]$ such that
$\ol{M'}=\ol{M}\cup_{\p\ol{M}}W$, where $(W; \bdry\ol{M}, \bdry\ol{M'})$ is an $h$-cobordism, and where the map $\ol{f'}$ extends $\ol{f}$ and has image in
$\p\ol{X}$.
Therefore $\Psi$ is well-defined.
It is now straightforward to see that $\Psi$ is a two-sided inverse of $\Phi$.
\end{proof}

The geometric structure set $\mcS^h_\TOP(\ol{X},\bdry\ol{X})$ can be identified with a version of Ranicki's algebraic structure group for
the same pair. The connective algebraic structure groups $\mcS_*^h$ are the homotopy groups of the homotopy cofiber of a assembly map $\alpha\connective$, and so they fit into an exact sequence of abelian groups \cite{Ranicki_TSO}:
\[
\cdots ~\to~ H_*(A,B; \bL\connective) ~\xrightarrow{\alpha\connective}~ L^{h}_*(A,B) ~\to~ \mcS_*^h(A,B) ~\to~ H_{*-1}(A,B; \bL\connective) ~\xrightarrow{\alpha\connective}~ \cdots
\]
where $\bL\connective$ is the 1-connective algebraic $L$-theory spectrum of the trivial group.
Exactly as in \cite{CDK}, for computations we shall use the non-connective, periodic analogue:
\[
\cdots ~\to~ H_*(A,B; \bL) ~\xrightarrow{\alpha}~ L_*^h(A,B) ~\to~ \mcS_*^{per,h}(A,B) ~\to~ H_{*-1}(A,B; \bL) ~\xrightarrow{\alpha}~ \cdots
\]
where $\bL$ is the 4-periodic algebraic $L$-theory spectrum of the trivial group.

\begin{rem}\label{rem:TSO}
Following Ranicki \cite[Thm.~18.5]{Ranicki_ALTM}, there is a pointed map
\[
s: \begin{tikzcd}\mcS_\TOP(\ol{X},\bdry\ol{X}) \rar{\approx} & \mcS^h_{n+1}(\ol{X},\bdry\ol{X}).\end{tikzcd}
\]
which is a bijection for $n > 5$.   The map $s$ is called the \emph{total surgery obstruction} for homotopy equivalences.
In our case, it is also a bijection when $n=5$, since each connected component of $\bdry\ol{X}$ is a 4-dimensional spherical space form, with finite fundamental group, and so can be included with the high-dimensional ($n>5$) case by Freedman--Quinn topological surgery \cite{FQ}.
Since $\overline{X}$ is $n$-dimensional, by the Atiyah--Hirzebruch spectral sequence, we obtain:
\[
H_{n+1}(\ol{X},\bdry\ol{X}; \bL/\bL\connective) = 0 \quad\text{and}\quad H_n(\ol{X},\bdry\ol{X}; \bL/\bL\connective) \subseteq \begin{cases}\Z & \text{if } n \text{ even}\\ \Z/2 & \text{if } n \text{ odd}.\end{cases}
\]
Hence there is an exact sequence:
\[\begin{tikzcd}
0 \rar & \mcS^h_{n+1}(\ol{X},\bdry\ol{X}) \rar & \mcS^{per,h}_{n+1}(\ol{X},\bdry\ol{X}) \rar & H_n(\ol{X},\bdry\ol{X}; \bL/\bL\connective).
\end{tikzcd}\]
At the end of the proof of Theorem~\ref{thm:main} in Section~\ref{sec:structures}, we will show that the last map is zero.  In the meantime, we will proceed to compute the non-connective algebraic structure group $\mcS^{per,h}_{n+1}(\ol{X},\bdry\ol{X})$.
\end{rem}

\section{Reduction from $h$ to $-\infty$ structure groups}\label{sec:lowering}

Our goal in this section is to replace the $h$ decoration by $-\infty$ in the structure group by showing $\mcS^{per,h}_{n+1}(\ol{X},\bdry\ol{X}) \cong \mcS^{per,-\infty}_{n+1}(\ol{X},\bdry\ol{X})$.  


A group is \emph{virtually cyclic} if there is a cyclic subgroup of finite index.
There is a well-known trichotomy of types of virtually cyclic groups $V$:
\begin{enumerate}
\item[(I)] $V$ is finite
\item[(II)] there is an exact sequence with $F$ finite and $C_\infty$ the infinite cyclic group:
\[\begin{tikzcd}
1 \rar & F \rar & V \rar & C_\infty \rar & 1
\end{tikzcd}\]
\item[(III)] there is an exact sequence with $F$ finite and $D_\infty$ the infinite dihedral group:
\[\begin{tikzcd}
1 \rar & F \rar & V \rar & D_\infty \rar & 1.
\end{tikzcd}\]
\end{enumerate}
For the rest of the paper, we consider the following increasing chain of classes $\mcF$:
\begin{itemize}
\item $1$ denotes the class of all trivial groups
\item $\fin$ denotes the class of all groups of type~(I)
\item $\fbc$ denotes the class of all groups of types~(I, II) --- the finite-by-cyclics
\item $\vc$ denotes the class of all groups of types~(I, II, III)
\item $\all$ denotes the class of all groups.
\end{itemize}
Given a group $G$ and one of the above five classes $\mcF$, we shall consider the \emph{family} $\mcF(G) := \{ H \in \mcF \mid H \leqslant G \}$ of subgroups of $G$.   Each family is closed
under conjugation and subgroups.
One says that \emph{$\G$ satisfies Property~$NM_{\mcF \subset\mcG}$} if every element $V$ of $\mcG(\G) - \mcF(\G)$ is contained in a unique maximal element $V_{max}$ of  \; $\mcG(\G) - \mcF(\G)$, and if, in addition
$V_{max}$ equals  its normalizer $N_\G(V_{max})$ in $\G$.

\begin{lem}\label{lem:surprise}
Suppose $\G$ is a group satisfying Hypothesis~(B).
If $V \in \vc(\G) - \fin(\G)$, then $V$ is isomorphic to either $C_\infty$ or $D_\infty$.
\end{lem}

\begin{proof}
By the trichotomy, the group $V$ contains a finite normal subgroup $F$ such that $V/F$ is isomorphic to either $C_\infty$ or $D_\infty$.
If $F\neq 1$, then $V$ is a subgroup of  $N_\G(F)$, contradicting (B).
So $F=1$.
Therefore $V \iso C_\infty $ or $V \iso D_\infty$.
\end{proof}

The maximal infinite dihedral subgroups are self-normalizing, as follows.

\begin{lem}\label{lem:selfnormalizing}
Let $\G$ be a group satisfying Hypothesis~(B).
\begin{enumerate} 
\item If $V \in \vc(\G) - \fbc(\G)$, then $N_\G(V) \in \vc(\G)$.
\item If $\G$ also satisfies Hypothesis~(C\tp ii), then $\G$ satisfies $NM_{\fbc\subset\vc}$.
\end{enumerate}
\end{lem}

\begin{proof}
(1)~By Lemma~\ref{lem:surprise}, $V$ is infinite dihedral.
That is, there exist subgroups $V_1, V_2 \subset V$ such that $V = V_1 * V_2$ and $V_1 \iso C_2 \iso V_2$.
There is an exact sequence
\[\begin{tikzcd}
1 \rar & C_\G(V) \rar & N_\G(V) \rar{p} & \Aut(V)
\end{tikzcd}\]
where $p(\gamma)$ is conjugation by $\gamma$.

First, note $C_\G(V) = C_\G(V_1) \cap C_\G(V_2)$.
By Hypothesis~(B), both $C_\G(V_1)$ and $C_\G(V_2)$ are finite.
Hence $C_\G(V)$ is finite.
Next, note $\Aut(V) = \mathrm{Inn}(V) \rtimes \langle\mathrm{sw}\rangle \cong (C_2 * C_2) \rtimes_{\mathrm{sw}} C_2$, where $\mathrm{sw}$ denotes the switch
automorphism on $V \cong C_2 * C_2$.
Thus, since $\Aut(V)$ contains an infinite cyclic subgroup of index $4$, it is virtually cyclic.
So the image $\mathrm{Im}(p)$ is virtually cyclic.
Therefore, since the kernel $\mathrm{Ker}(p) = C_\G(V)$ is finite, the normalizer $N_\G(V)$ is virtually cyclic.

\noindent (2)~Any $V \in \vc(\G) - \fbc(\G)$ is infinite dihedral by part (1).  By Hypothesis~(C\tp ii), $V$ is contained in a unique infinite dihedral group, which we will call $V_{max}$.  By Lemma~\ref{lem:surprise}, $V_{max}$ is also a maximal virtually cyclic subgroup Thus by part (1), it self-normalizing.   Hence $NM_{\fbc\subset\vc}$ holds.
\end{proof}

Observe that Property~$NM_{1 \subset \fin}$ implies Hypothesis~(B).
A partial converse is:
\begin{cor}\label{cor:NM}
Suppose $\G$ satisfies Hypotheses~(AB).
If $\mcS(\G)$ is nonempty, then $\G$ satisfies Property~$NM_{1 \subset\fin}$.
\end{cor}

\begin{proof}
There exists a contractible manifold $M$ equipped with a cocompact proper $\G$-action.
Let $H \in \fin(\G) - 1$.
By Proposition~\ref{prop:onepoint}, the fixed set $M^H$ is a single point, say $\{x\}$.
Since the action $\G \curvearrowright M$ is proper, the isotropy group $H_{max} := \G_x$ is finite.
Let $K \in \fin(\G)$ contain $H$.
By Proposition~\ref{prop:onepoint} again, $M^K \subseteq M^H$ is non-empty.
We must have $M^K = \{x\}$. Hence $K \subseteq H_{max}$.
Thus $\G$ satisfies $M_{1\subset\fin}$.

By Hypothesis~(B), $N_\G(H_{max}) \in \fin(\G)$.  By Proposition~\ref{prop:onepoint}, it fixes a single point, which is $\{x\}$.   Hence $N_\G(H_{max}) = H_{max}$.  Thus $\G$ satisfies $NM_{1 \subset\fin}$.
\end{proof}


For any group $G$, ring $R$, and integer $q$, recall the generalized Whitehead groups
\[
\Wh^R_q(G) ~:=~ H^G_q(E_\all G ,E_{1} G; \ul{\bK}_R).
\]

A direct sum decomposition of Whitehead groups is in \cite[Theorem~5.1(d)]{DL2}.

\begin{thm}\label{thm:sumWh}
Let $\G$ be a group satisfying Property~$NM_{1 \subset \fin}$ and the Farrell--Jones Conjecture in algebraic $K$-theory.
Then, for each $q \in \Z$, the inclusion-induced map is an isomorphism:
\[
\begin{tikzcd}\displaystyle\bigoplus_{(F) \in (\mfin)(\G)} \Wh^\Z_q(F)  \rar & \Wh^\Z_q(\G).\end{tikzcd}
\]
Here $(\mfin)(\G)$ is the set of conjugacy classes of maximal finite subgroups of $\G$.
\end{thm}

For any connected space $Z$, we define $\Wh_*(Z) := \Wh^\Z_*(\pi_1 Z)$.

\begin{cor}\label{cor:sumWh}
Let $\G$ be a group satisfying Hypotheses~(AB) and the Farrell--Jones Conjecture in algebraic K-theory.
Suppose $X$ is a contractible manifold of dimension $\geqslant 3$ equipped with an effective cocompact proper $\G$-action.
Then, for each $q \in \Z$, the inclusion-induced map is an isomorphism:
\[
A_q : \begin{tikzcd}\displaystyle\bigoplus_{j=1}^m \Wh_q(\bdry_j\ol{X}) \rar & \Wh_q(\ol{X}).\end{tikzcd}
\]
\end{cor}

\begin{proof}
Since $\G$ satisfies Hypotheses~(AB), by Corollary~\ref{cor:NM}, $\G$ satisfies $NM_{1\subset\fin}$.
By Theorem~\ref{thm:sumWh}, the induced map is an isomorphism:
\[
\begin{tikzcd}\displaystyle\bigoplus_{(F) \in (\mfin)(\G)} \Wh^\Z_q(F) \rar{\iso} & \Wh^\Z_q(\G).\end{tikzcd}
\]
Since $X$ is a $\G$-space so that for $1 \not = H < \G$, $X^H$ is empty for $H$ infinite and $X^H$ is a point for $H$ finite, the following function is a bijection:
\[
\alpha: X_{sing}/\G \longrightarrow (\mfin)(\G) ~;~ q \longmapsto \{\G_p \mid \G p = q\}.
\]
Since $\dim X \geqslant 3$, there are isomorphisms of  fundamental groups
\[\begin{tikzcd}
\pi_1 \ol{X} & \arrow{l}[swap]{\cong} \pi_1(X_{free}/\G) \rar{\cong} & \G.
\end{tikzcd}\]
Finally, the map $e$ from \eqref{e_map} induces a bijection
\[
\beta: \{ \pi_1 \partial_i \ol X\}_{i=1}^m \longrightarrow X_{sing}/\G
\]
so that the bijection $\alpha \circ \beta$ is given  on representatives by group isomorphisms which are compatible with the inclusion maps to $\G$.
\end{proof}

\begin{cor}\label{cor:h_inf}
Let $(X,\G)$ be as above.
The forgetful map is an isomorphism:
\[
\begin{tikzcd}\mcS_*^{per,h}(\ol{X}, \bdry\ol{X}) \rar & \mcS_*^{per,\infdec}(\ol{X}, \bdry\ol{X}).\end{tikzcd}
\]
\end{cor}

\begin{proof}
For each $q \leqslant 1$, there is a commutative diagram with exact rows and $\langle 1 \rangle = h$:
\[\begin{tikzcd}
\cdots \rar & \mcS_{*+1}^{per,\langle q \rangle}(\ol{X}, \bdry\ol{X}) \rar \dar & H_*(\ol{X}, \bdry\ol{X}; \bL) \rar \dar{=} & L^{\langle q \rangle}_*(\ol{X}, \bdry\ol{X}) \rar \dar & \cdots\\
\cdots \rar & \mcS_{*+1}^{per,\langle q-1 \rangle}(\ol{X}, \bdry\ol{X}) \rar & H_*(\ol{X}, \bdry\ol{X}; \bL) \rar & L^{\langle q-1 \rangle}_*(\ol{X}, \bdry\ol{X}) \rar & \cdots
\end{tikzcd}\]
By \cite[Proposition~2.5.1]{Ranicki_ESATS}, there is a diagram with exact rows and columns:
\[\small\begin{tikzcd}
{} & {} \dar & {} \dar & {} \dar & {}\\
\cdots \rar & \displaystyle\bigoplus_{j=1}^m L^{\langle q\rangle}_*(\bdry_j\ol{X}) \rar \dar & L^{\langle q\rangle}_*(\ol{X}) \rar \dar & L^{\langle q\rangle}_*(\ol{X},\bdry\ol{X}) \rar \arrow{d}{f_q} & \cdots\\
\cdots \rar & \displaystyle\bigoplus_{j=1}^m L^{\langle {q-1}\rangle}_*(\bdry_j\ol{X}) \rar \dar & L^{\langle {q-1}\rangle}_*(\ol{X}) \rar \dar & L^{\langle {q-1}\rangle}_*(\ol{X},\bdry\ol{X}) \rar \dar & \cdots\\
\cdots \rar & \displaystyle\bigoplus_{j=1}^m \wh{H}^*(C_2; \Wh_{q-1}(\bdry_j\ol{X})) \rar{(A_{q-1})*} \dar & \wh{H}^*(C_2; \Wh_{q-1}(\ol{X})) \rar \dar & \wh{H}^*(C_2; A_{q-1}) \rar \dar & \cdots\\
{} & {} & {} & {} & {}
\end{tikzcd}\]
Here $A_{q-1}$ is the isomorphism of Corollary~\ref{cor:sumWh}
and $\wh{H}^*(C_2;A_{q-1})$ is a certain group defined by Ranicki \cite[p.~166]{Ranicki_ESATS}, which must vanish by exactness.
Then, via the commutative diagram, $f_q$ is an isomorphism.
So, by the Five Lemma, $\mcS_{*}^{per,\langle q \rangle}(\ol{X}, \bdry\ol{X}) \to \mcS_{*}^{per,\langle q-1 \rangle}(\ol{X}, \bdry\ol{X})$ is an isomorphism.  The result follows, since
\[
\mcS_{*}^{per,\infdec}(\ol{X}, \bdry\ol{X}) ~=~ \colim_{q \to -\infty} \mcS_{*}^{per,\langle q \rangle}(\ol{X}, \bdry\ol{X}).
\qedhere
\]
\end{proof}

\section{Calculation of the $-\infty$ structure groups}\label{sec:structures}

For a group $\G$ with orientation character $w : \G  \to \{\pm 1\}$, let $\ul{\bL} := \ul{\bL}_\Z^\infdec :\Or(\G) \to \Spectra$ be the  Davis--L\"uck functor \cite{DL1}.  This defines a $\G$-homology theory which assigns an abelian group $H^\G_n(A,B; \ul\bL)$ to a pair of $\G$-spaces and an integer.  The ``coefficients" are given by $H^\G_n(\G/H,\varnothing; \bL) = \pi_n\ul{\bL} (\G/H) \cong L^{-\infty}_n(\Z H,w)$.

Our goal in this section is, for a group $\G$ satisfying Hypotheses (ABC\tp), to identify $\mcS^{per,\infdec}_{n+1}(\ol{X}, \bdry\ol{X})$ with $H^\G_*(E_\all \G, E_\fin\G; \ul\bL)$ and to compute this in terms of UNil-groups.
As a byproduct of the computation we will see that the map $\mcS^\TOP(\ol{X},\bdry\ol{X}) \cong \mcS^h_{n+1}(\ol X,\bdry\ol{X}) \to \mcS^{per,h}_{n+1}(\ol{X},\bdry\ol{X}) \cong \mcS^{per,\infdec}_{n+1}(\ol{X},\bdry\ol{X})$ is a bijection, as promised in Remark~\ref{rem:TSO}.
This will complete the proof of Theorem \ref{thm:main}.

\begin{lem}\label{lem:split}
Suppose $\G$ satisfy Hypothesis (ABC\tp) with $\dim X \geqslant 5$.   Let $\ol{X}$ be a compact $\partial$-manifold with interior $X_{free}/\G$.   
There is a commutative diagram with long exact rows and vertical isomorphisms:
\[\small\begin{tikzcd}
H_*^\G(X,X_{free}; \ul{\bL}) \rar \dar &
H_*^\G(cX, X_{free};\ul{\bL}) \rar \dar &
H_*^\G(cX, X; \ul{\bL}) \rar \dar &
H_{*-1}^\G(X,X_{free}; \ul{\bL}) \dar\\
\mcS^{per,\infdec}_*(\bdry\ol{X}) \rar &
\mcS^{per,\infdec}_*(\ol{X}) \rar &
\mcS^{per,\infdec}_*(\ol{X},\bdry\ol{X}) \rar &
\mcS^{per,\infdec}_{*-1}(\bdry\ol{X}).
\end{tikzcd}\]
Thus $\mcS^{per,\infdec}_*(\ol{X},\bdry\ol{X}) \cong H^\G_*(E_\all \G, E_\fin\G; \ul\bL)$.  
\end{lem}

\begin{proof}
The argument is closely analogous to that of \cite[Lemma~4.2]{CDK}.
\end{proof}

The next lemma allows us to simplify our families.

\begin{lem}\label{lem:change_of_family} 
Let $\G$ be a group.
\begin{enumerate}

\item \label{lem:fbcfin} $H_*^\G(E_\fbc \G, E_\fin \G; \ul{\bL}) = 0$.

\item \label{lem:allvc}  $H_*^\G(E_\all \G, E_\vc \G; \ul{\bL}) = 0$ if the Farrell--Jones Conjecture  in $L$-theory holds for $\G$.

\end{enumerate}
\end{lem}

\begin{rem}
Part~\eqref{lem:allvc} is simply the modern statement of the Farrell--Jones Conjecture.
Part~\eqref{lem:fbcfin} should be contrasted with the corresponding result in $K$-theory: $H_*^\G(E_\vc\G, E_\fbc \G; \ul{\bK}_R)=0$ for any group $\G$ and for any ring $R$ (see \cite{dqr}, also \cite{dkr}).
Part~\eqref{lem:fbcfin} is given as Lemma 4.2 of \cite{Lueck_Heisenberg}, however the proof lacks some details, so we give a complete proof in a special case.
\end{rem}

\begin{proof}[Proof of Lemma \ref{lem:change_of_family} \eqref{lem:fbcfin} in the case where $\G$ satisfies Hypothesis (B)]  Suppose $\G$ is a group satisfying Hypothesis~(B).
By the Farrell--Jones Transitivity Principle (see~\cite[Theorem~65]{LR_survey}), and by Lemma~\ref{lem:surprise}, it suffices to show:
\[
 H_*^{C_\infty}(E_\fbc C_\infty, E_\fin C_\infty; \ul{\bL}) ~=~ 0.
\]
That is, we must show the following Davis--L\"uck assembly map is an isomorphism:
\[
H_*^{C_\infty}(\R; \ul{\bL}) ~\longrightarrow~ H_*^{C_\infty}(\ast; \ul{\bL}).
\]
Since $\R$ is a simply connected, free $C_\infty$-CW complex, by \cite[Theorem~B.1]{CDK} (see also \cite{HambletonPedersen}), this is equivalent to the following Quinn--Ranicki assembly map being an isomorphism:
\[
H_*(S^1; \bL(\Z)) ~\longrightarrow~ L_*(\Z[C_\infty]).
\]
In other words, we must show the vanishing of Ranicki's algebraic structure groups:
\[
\mcS_*^{per}(S^1) ~=~ 0.
\]
For all $n \geqslant 4$, by \cite[Theorem~18.5]{Ranicki_ALTM}, there is a bijection:
\[
\mcS_\TOP(S^1 \times D^n \rel \bdry) ~\longrightarrow~ \mcS_{n+2}(S^1 \times D^n) ~\iso~ \mcS_{n+2}(S^1).
\]
For all $n \geqslant 4$, by the Farrell--Hsiang rigidity theorem~\cite[Theorem~4.1]{FH_spaceform}, the structure set $\mcS_\TOP(S^1 \times D^n \rel \bdry)$ is a singleton.
Thus $\mcS^{per}_k(S^1) =\mcS_k(S^1)=0$ for all $k \geqslant 6$.
These structure groups are 4-periodic, so $\mcS_*^{per}(S^1)=0$.
\end{proof}

Lastly, we recall the identification of \cite[Lemma~4.6]{CDK}, done for $L$-theory.

\begin{lem}[Connolly--Davis--Khan]\label{lem:UNil_relhom}
Let $n \in \Z$ and write $\eps := (-1)^n$.
The following composite map, starting with Cappell's inclusion, is an isomorphism:
\[\begin{tikzcd}
\UNil_{n+\eps}(\Z;\Z,\Z) \rar[dashed] \dar[tail] & H_{n+1}^{D_\infty}(E_\all D_\infty,E_\fin D_\infty; \ul{\bL})\\
L_{n+1}(C_2^\eps * C_2^\eps) \rar{\iso} & H_{n+1}^{D_\infty}(E_\all D_\infty; \ul{\bL}) \uar
\end{tikzcd}\]
\end{lem}

The next lemma uses excision and is a more abstract version of 
 \cite[Lemma~4.5]{CDK}.

\begin{lem}\label{lem:directsum}
Suppose $\G$ is a group satisfying Hypotheses~(BC\tp ii).
Let $n \in \Z$.
Assume the orientation character $\omega: \G \to \{\pm 1\}$ evaluates to $\eps := (-1)^n$ on all elements of order two.
Then there is an inclusion-induced isomorphism
\[
\begin{tikzcd}\displaystyle\bigoplus_{(\midd)(\G)} \UNil_{n+\eps}(\Z;\Z,\Z) \rar & H_{n+1}^\G(E_\all \G, E_\fin \G; \ul{\bL}).\end{tikzcd}
\]
\end{lem}

\begin{proof}
By Lemma~\ref{lem:change_of_family}, there are induced isomorphisms
\[
H_*^\G(E_\all \G,E_\fin \G; \ul{\bL}) \xrightarrow{~\cong~} H_*^\G(E_\all \G,E_\fbc \G; \ul{\bL}) \xleftarrow{~\cong~} H_*^\G(E_\vc \G,E_\fbc \G; \ul{\bL}).
\]
Write $(\mvc)(\G)$ for the set of conjugacy classes of maximal virtually cyclic subgroups of $\G$.
By Lemma~\ref{lem:selfnormalizing}, $\G$ satisfies $NM_{\fbc \subset \vc}$.
Then, by \cite[Corollary~2.8]{LW} and the Induction Axiom \cite[Prop.~157, Thm.~158(i)]{LR_survey}, there is an isomorphism
\[
\bigoplus_{(D) \in (\mvc)(\G)} H_*^D(E_\vc D,E_\fbc D; \ul{\bL}) ~\longrightarrow~ H_*^\G(E_\vc \G,E_\fbc \G; \ul{\bL}).
\]
By Lemma~\ref{lem:surprise}, we may replace the index set $(\mvc)(\G)$ by $(\midd)(\G)$.
Let $(D) \in (\midd)(\G)$. By Lemma~\ref{lem:change_of_family} again, there are induced isomorphisms
\[
H_*^D(E_\all D,E_\fin D; \ul{\bL}) \xrightarrow{~\cong~} H_*^D(E_\all D,E_\fbc D; \ul{\bL}) \xleftarrow{~\cong~} H_*^D(E_\vc \G,E_\fbc D; \ul{\bL}).
\]
Thus there is an isomorphism
\[
\bigoplus_{(D) \in (\midd)(\G)} H_*^D(E_\all D,E_\fin D; \ul{\bL}) ~\longrightarrow~ H_*^\G(E_\all \G,E_\fin \G; \ul{\bL}).
\]
Finally,  Lemma~\ref{lem:UNil_relhom} gives the desired conclusion.
\end{proof}

Now we put the pieces together and prove our Main Theorem.

\begin{proof}[Proof of Theorem~\ref{thm:main}]
By Theorem~\ref{thm:equi-iso} (whose proof was given in Section~\ref{sec:eq_to_iso}), 
the  forgetful map is a bijection:
\[
\mcS^{iso}(X,\G) ~\longrightarrow~ \mcS(X,\G) ~\longrightarrow~ \mcS(\G).
\]
Since Hypothesis~(C\tp iii) holds for $K$-theory, by Lemma~\ref{lem:compactification}, there is a bijection:
\[
\mcS^h_\TOP(\ol{X},\bdry\ol{X}) ~\longrightarrow~ \mcS^{iso}(X,\G).
\]
By Remark~\ref{rem:TSO} and Corollary~\ref{cor:h_inf}, the following composition is injective:
\begin{equation}\label{eqn:injective}
\mcS^h_\TOP(\ol{X},\bdry\ol{X}) ~\longrightarrow~ \mcS^h_{n+1}(\ol{X},\bdry\ol{X}) ~\longrightarrow~ \mcS^{per,h}_{n+1}(\ol{X},\bdry\ol{X}) ~\longrightarrow~ \mcS^{per,\infdec}_{n+1}(\ol{X},\bdry\ol{X}).
\end{equation}
Since Hypothesis~(C\tp iii) holds for $L$-theory, by Lemma~\ref{lem:split}, there is an isomorphism:
\[
H^\G_{n+1}(E_\all \G,E_\fin \G; \ul{\bL}) ~\longrightarrow~ \mcS^{per,\infdec}_{n+1}(\ol{X},\bdry\ol{X}).
\]
Consider the  diagram:
\[
\begin{tikzcd}
\displaystyle\bigoplus \UNil_{n+\eps}(\Z;\Z,\Z) \rar \dar &
L_{n+1}^h(\G,\omega) \rar \dar & 
\mcS^h_\TOP(\ol{X},\bdry\ol{X}) \dar\\
\displaystyle\bigoplus H^D_{n+1}(E_\all D,E_\fin D; \ul{\bL}) \rar &
H^\G_{n+1}(E_\all\G,E_\fin\G; \ul{\bL}) \rar &
\mcS^{per,\infdec}_{n+1}(\ol{X},\bdry\ol{X}).
\end{tikzcd}
\]
The definitions of all the maps should be clear to the reader.
Assuming momentarily that the diagram commutes, we complete the proof of Theorem~\ref{thm:main}.
We have already shown that the composition from the upper left of the diagram to the lower left, to the lower middle, and finally to the lower right is a bijection.
Furthermore, Remark~\ref{rem:TSO} and Corollary~\ref{cor:h_inf} show that the right vertical map is an injection.
It follows that the right vertical map is a bijection.
This concludes the proof of Theorem~\ref{thm:main}.
Note that, as a byproduct, we have also shown that
\[
\begin{tikzcd}\mcS^h_{n+1}(\ol{X},\bdry\ol{X}) \rar & \mcS^{per,h}_{n+1}(\ol{X},\bdry\ol{X})\end{tikzcd}
\]
is an isomorphism, as promised in Remark~\ref{rem:TSO}.

Finally, we need to argue that the above diagram is commutative.
The left hand square commutes by the definition of the maps.
The commutativity of right square follows from three commutative squares, the first of which is:
\[
\begin{tikzcd}
L_{n+1}^h(\G,\omega) \rar \dar & 
\mcS^h_\TOP(\ol{X}) \dar\\
H^\G_{n+1}(E_\all\G; \ul{\bL}) \rar &  \mcS^{per,\infdec}_{n+1}(\ol{X}).
\end{tikzcd}
\]
Here, the bottom map is the composite of $H^\G_{n+1}(E_\all\G; \ul{\bL}) \to H^\G_{n+1}(E_\all\G, \ol{X}; \ul{\bL})$ and the identification $H^\G_{n+1}(E_\all\G, \ol{X}; \ul{\bL}) \cong \mcS^{per,\infdec}_{n+1}(\ol{X})$ of \cite[Appendix~B]{CDK}.
The commutativity of the square then follows from \cite[Appendix~B]{CDK} and Ranicki's identification of algebraic and geometric structure sets \cite[Theorem~18.5]{Ranicki_ALTM}.

The second square is 
\[
\begin{tikzcd}
H^\G_{n+1}(E_\all\G; \ul{\bL}) \rar \dar &  \mcS^{per,\infdec}_{n+1}(\ol{X}) \dar\\
H^\G_{n+1}(E_\all\G,E_\fin\G; \ul{\bL}) \rar &
\mcS^{per,\infdec}_{n+1}(\ol{X},\bdry\ol{X})
\end{tikzcd}
\]
and the third square is
\[
\begin{tikzcd}
\mcS^h_\TOP(\ol{X}) \rar \dar &  \mcS^h_\TOP(\ol{X},\bdry\ol{X})\dar\\
\mcS^{per,\infdec}_{n+1}(\ol{X}) \rar & \mcS^{per,\infdec}_{n+1}(\ol{X},\bdry\ol{X}).
\end{tikzcd}
\]

These last two squares commute by naturality of all the maps involved.
\end{proof}

The map $\partial$ in Theorem \ref{thm:main} is given by applying the Wall realization theorem to the manifold $\ol{X}$.
In particular, if one element of $\mcS(\G)$ admits a smooth structure then so do all elements of $\mcS(\G)$.

\section{Equivariant $h$-cobordisms} \label{sec:eq_h}

\begin{proof}[Proof of Theorem~\ref{thm:hcobordism_rigidity}]
By Hypothesis~(A), there exists a torsion-free subgroup $\G_0$ of $\G$ of finite index.
By intersecting $\G_0$ with its finitely many conjugates, we may assume $\G_0$ is normal in $\G$.
Write $G := \G/\G_0$ for the finite quotient group.

Let $(W;M,M')$ be a locally flat $\G$-isovariant $h$-cobordism.
By \cite[Corollary~1.6]{Quinn_HSS} applied to the locally flat $G$-isovariant $h$-cobordism $(W/\G_0; M/\G_0, M'/\G_0)$, we obtain that $(W/\G; M/\G, M'/\G)$ is an $h$-cobordism of manifold homotopically stratified spaces.
Observe $M/\G$ has only two strata:
\[
M_n := M_{free}/\G \quad\text{and}\quad \delta M_n = M_{sing}/\G.
\]
Quinn's stratified torsion for the quotient $h$-cobordism is an element of his obstruction group $H_1(c\delta M_n, \delta M_n; \bP(q_n))$.
(This is a relative homology group with cosheaf coefficients in non-connective pseudo-isotopy theory; see \cite[Definition~8.1]{Quinn_EndsII}.)
This obstruction group fits into the exact sequence of the pair that simplifies to:
\begin{multline*}
\bigoplus_{\G x \in \delta M_n} \Wh_1(\G_x) \xrightarrow{~A_1~}
\Wh_1(\G) \longra
H_1(c\delta M_n, \delta M_n; \bP(q_n))\\
\xrightarrow{~\bdry~}
\displaystyle\bigoplus_{\G x \in \delta M_n} \Wh_0(\G_x) \xrightarrow{~A_0~}
\Wh_0(\G).
\end{multline*}
Here, for any point $x$ in the singular set $M_{sing}$, its orbit and isotropy group are:
\[
\G x := \{gx ~|~ g \in \G\} \quad\text{and}\quad \G_x := \{g \in \G ~|~ gx=x \}.
\]

By Corollary~\ref{cor:sumWh}, $A_1$ is surjective and $A_0$ is injective.
Then Quinn's obstruction group vanishes: $H_1(c\delta M_n, \delta M_n; \bP(q_n)) = 0$.
So, by \cite[Theorem~1.8]{Quinn_HSS}, $W/\G$ is a stratified product cobordism. 
  
   Let $f:M/\G\times(I; 0, 1)\to (W/\G; M/\G, M'/\G)$ be a homeomorphism of stratified spaces.  Each stratum of $W$ is a covering space of the corresponding stratum in $W/\G$. So for each $x\in M$, the path,
   \[
   f_x: I \to W/\G;\qquad f_x(t):= f(x,t)
       \]  
         lifts uniquely to a path in $F_x:I\to W$, with  $F_x(0)=x$. This specifies a bijection, 
   \[
   F:M\times I\to W;\qquad \qquad
  F(x,t):=F_x(t)
  \]
  
  Granting for a moment that $F$ is continuous, it follows that $F$ is a $\G$-homeomorphism and $W$ is a product cobordism. In particular, $M$ and $M'$ are $\G$-homeomorphic.

  We need only prove that $F$ is continuous at each point of $ M_{sing}\times I$. 
  So let $(x_0, t)\in M_{sing}\times I$. Set $y_0= F(x_0, t)$. 
  
  A basis of neighborhoods of $y_0$ in $W$ is given by the collection of sets of form
  \[
  U_0\cap \rho^{-1}(\mcO), 
 \] 
 where:
  \begin{enumerate}
  \item $\mcO$ is an open set of $I$ containing $t$, and 
  \[
  \rho:= p_2\circ f^{-1}\circ\pi^W: W\to W/\G\to M/\G\times I \to I
  \]
  \item $U_0$ is a single component of  a $\G$-invariant open set $U$ of $W$ such that 
  \[
  F(\{x_0\}\times I)= U_0\cap W_{sing}.
  \]
  \end{enumerate}
  Given such $U_0$ and $\mcO$ then, we seek a neighborhood $N_0$ of $x_0$ in $M$ such that
  \[
  F(N_0\times \mcO)\subset U_0\cap \rho^{-1}(\mcO).
  \]
  Let $\mcN$ be a connected open neighborhood of $\pi^M(x_0)$ in $M/\G$ such that 
  \[
  f(\mcN\times I)\subset \pi^W(U) \text{\; in } W/\G. 
  \]
 Let $N_0$ be the component of $(\pi^M)^{-1}(\mcN)$ containing $x_0$. Then
 $F(N_0\times I)\subset U$, and $F(N_0\times I)$ is connected (since $F(N_0\times 0)=N_0$ and $N_0$ is connected).
 Therefore $F(N_0\times I)\subset U_0$.  Also $F(N_0\times \mcO)\subset \rho^{-1}(\mcO).$ So
 \[
 F(N_0\times \mcO)\subset U_0\cap \rho^{-1}(\mcO),
 \]
as required. This proves $F$ is continuous.  
\end{proof}

\section{Examples}  \label{sec:examples}

In this section we give examples of groups satisfying the hypotheses of Theorem \ref{thm:main}, and hence groups where the question of equivariant rigidity is completely answered.

We first give some examples of groups $\G$ satisfying Hypotheses~(ABC).
First and foremost, we mention the group $\G = \Z^n \rtimes_{-1} C_2$ (with $n \geqslant 5$) which was the subject of our previous paper \cite{CDK}.   Our second example is the generalization $\G = \Z^n \rtimes_{\rho} C_m$ where $n \geqslant 5$ and where the $C_m$-action given by $\rho$ is free on $\Z^n - \{0\}$.
This generalization requires the $K$-theory analysis of our current paper, and if $m$ is odd, give examples of groups satisfying the rigidity of Corollary~\ref{cor:2torsionfree}.

Next suppose $G$ is a discrete cocompact subgroup of the isometries of the hyperbolic or Euclidean plane without reflections, thus the normalizers of nontrivial finite subgroups are finite.
This is a subject of classical interest and there is a ready supply of examples.
Note the product $\G = G^r$ for $r \geqslant 3$ satisfies Hypotheses~(ABC).

The announcement \cite{Khan_announce} gave examples satisfying Hypotheses~(ABC\tp) but not Hypotheses~(ABC).
The first of these constructions is due to Davis--Januskiewicz, and it is obtained from a triangulated homology sphere $\Sigma$ with nontrivial fundamental group.
The dual cones of the triangulation decompose $\Sigma$ as a union of contractible closed subcomplexes (``mirrors'') any of whose intersections are either empty or contractible (``submirrors'').
This structure gives a right-angled Coxeter group $W$.
The reflection trick of Mike Davis allows the construction of a contractible manifold $X$ obtained from $\Sigma$ on which a subgroup $\G$ of $W$ acts pseudo-freely and has 2-torsion.
This $X$ is not homeomorphic to Euclidean space; see \cite[Example~3.2]{Khan_announce}.
A second example is given there, derived from a Heisenberg-type group.
Unlike the above examples, it is not $\CAT(0)$, but does satisfy Hypotheses~(ABC\tp).

Finally, following a suggestion of the referee, we gives examples based on Gromov's technique of hyperbolization \cite[Section 3.4]{Gromov}.

\begin{lem}
 For any pseudo-free PL-action of a finite group $G$ on a closed PL-manifold $K$, there is a effective, cocompact, proper, isometric PL-action of a discrete group $\G$ on a complete CAT(0) PL-manifold $X$ so that the following conditions are satisfied.
 
\begin{enumerate}

\item  There is a short exact sequence 
$$
1 \to \G_0 \to \G \to G \to 1
$$
of  groups with $\G_0$ torsion-free.

\item The dimensions of $X$ and $K$ are equal and there is an $G$-isovariant map $X/\G_0 \to K$.

\item  The action of $\G$ on $X$ satisfies hypotheses (ABC\tp) of the introduction and hence $\G$ satisfies the conclusion of  Theorem \ref{thm:main}.

\end{enumerate}

\end{lem}

\begin{proof}

We will use the hyperbolization of Davis--Januszkiewicz \cite{DJ}.
This hyperbolization is an assignment which, for every finite simplicial complex $K$, assigns a compact polyhedron $h(K)$ which is a locally CAT(0), piecewise Euclidean, geodesic metric space and a PL-map $p(K) : h(K) \to K$ and for every simplicial embedding $f : K \to L$, assigns  an isometric PL-embedding onto a totally geodesic subpolyhedron $h(f): h(K) \to h(L)$.
It satisfies the following properties:
 
\begin{itemize}
 
\item (Functoriality)  The functorial identities $h(\id_K) = \id_{h(K)}$ and $h(f \circ g) = h(f) \circ h(g)$ and the natural transformation identity $p(L) \circ h(f) = f \circ p(K)$.

\item (Local structure)  If $\sigma^n$ is a closed $n$-simplex of $K$ then $h(\sigma^n)$ is a compact  $n$-manifold with boundary and the link of $h(\sigma^n)$ in $h(K)$ is PL-homeomorphic to the link of $\sigma^n$ in $K$.

\item $h(\text{point}) = \text{point}$.

\end{itemize}

Consider a pseudo-free PL-action of a finite group $G$ on a closed, PL-manifold $K$.   After  subdivisions, assume that $K$ is a simplicial complex, that invariant simplicies are fixed and that the star of the singular set is a regular neighborhood of the singular set.
Functoriality gives a $G$-action by PL-isometries on the hyperbolization $h(K)$.   Note that $h(K)$ is an aspherical manifold: it is a manifold whose dimension equals that of $K$ by the local structure property and it is aspherical since its universal cover is a simply-connected complete CAT(0) space, hence contractible.  Naturality shows that $p(K) : h(K) \to K$ is $G$-isovariant.  Let $\Gamma$ be the group of homeomorphisms of the universal cover $X$ of $h(K)$ which cover elements of $G$ and let $\Gamma_0$ be the fundamental group of $X$.    Then clearly the $\Gamma$-action on $X$ is  
effective, cocompact, proper, isometric, and PL.   Conditions (1) and (2) are satisfied.   

Now we turn to hypotheses (ABC\tp).   Part (1) shows hypothesis (A).   Since the hyperbolization of a $0$-simplex is a point by the local structure property, the $G$-action on $h(K)$ is pseudo-free, hence so is the $\Gamma$-action on $X$.   Proposition \ref{prop:onepoint} then shows that hypothesis (B) holds.  To see that hypothesis (C\tp i) holds,  note that $X_{free}/\G = h(K)_{free}/G = (h(K) - h(K)_{sing})/G \simeq (h(K) - \text{int star }h(K)_{sing})/G$, which is a finite complex.  

Hypothesis (C\tp ii) states that every infinite dihedral subgroup of $\G$ lies in a unique maximal infinite dihedral subgroup.   Let $\Delta$ be an infinite dihedral subgroup of $\G$ and let $\Delta =  \langle t \rangle \rtimes  \langle a \rangle$  where $t$ has infinite order and $a$ has order 2.  Let $x_n \in X$ be the fixed point of the involution $t^na \in \Delta \subset \G$. Let
\[
\gamma = \bigcup_{n\in \Z} \overline{x_n x_{n+1}},
\]
where $\overline{x_n x_{n+1}}$ is the geodesic line segment joining the two points (there is a unique geodesic segment connecting any two points, since $X$ is CAT(0)).   We now claim that for any $n$ and $k$, one has $x_{n+k} \in \overline{x_n x_{n+2k}}$.   This holds since $t^{n+k} a$ interchanges the end points of the geodesic segment $\overline{x_n x_{n+2k}}$ and must leave the midpoint invariant.   Since the fixed sets of all involutions are singletons, $x_{n+k}$ must be the midpoint of the geodesic segment.   The claim implies that $\gamma$ is a geodesic and is invariant under $\Delta$.   
Let $\wh \Delta = \{ g \in \G \mid g(\gamma) = \gamma \}$.   Then $\wh \Delta$ is the unique maximal infinite dihedral subgroup containing $\Delta$.

Hypothesis~(C\tp iii) holds by \cite{BL_CAT0} and \cite{Wegner}.
\end{proof}

Here is an example where the lemma applies.   Let $G$ be a space form group, that is, a finite group so that for every prime $p$, all subgroups of order $p^2$ and all subgroups of order $2p$ are cyclic.   Then according to Madsen--Thomas--Wall \cite{MTW}, $G$ acts freely and smoothly on a sphere $S^n$ for some $n$.   By triangulating the quotient, we can assume this action is PL.   Then $G$ acts semifreely on the suspension $K = \Sigma S^n$ with a two-point fixed set.  

This is of interest because the isotropy groups of any pseudo-free PL action on a manifold are space form groups, and the construction above shows that all such groups arise as isotropy groups.

\providecommand{\bysame}{\leavevmode\hbox to3em{\hrulefill}\thinspace}
\providecommand{\MR}{\relax\ifhmode\unskip\space\fi MR }
\providecommand{\MRhref}[2]{%
  \href{http://www.ams.org/mathscinet-getitem?mr=#1}{#2}
}
\providecommand{\href}[2]{#2}

\end{document}